\documentclass{amsart} 
\usepackage{amsmath}
\usepackage{paralist}
\usepackage[misc]{ifsym}
\usepackage{epsfig} 
\usepackage{epstopdf} 
\usepackage[hidelinks]{hyperref}
\allowdisplaybreaks
\usepackage[all,cmtip]{xy}

\newtheorem{theorem}{Theorem}[section]
\newtheorem{corollary}[theorem]{Corollary}

\newtheorem{lemma}[theorem]{Lemma}
\newtheorem{proposition}[theorem]{Proposition}

\theoremstyle{definition}
\newtheorem{definition}[theorem]{Definition}
\newtheorem{remark}[theorem]{Remark}
\newtheorem{example}[theorem]{Example}

\newcommand{\lra}{\longrightarrow}
\newcommand{\rr}{\mathbb{R}}

\newcommand{\frakA}{\mathfrak A}
\newcommand{\frakg}{\mathfrak g}
\newcommand{\frakS}{\mathfrak S}
\newcommand{\bF}{\mathbf F}
\newcommand{\Lie}{\operatorname{Lie}}
\newcommand{\Ad}{\operatorname{Ad}}

\newcommand{\calA}{\mathcal A}
\newcommand{\calL}{\mathcal L}

\newcommand{\calX}{\mathcal X}
\newcommand{\calY}{\mathcal Y}
\newcommand{\subs}{\subset}
\newcommand{\pr}{\operatorname{pr}}
\newcommand{\Lra}{\Longrightarrow}
\newcommand{\norm}[1]{\lVert#1\rVert}

\bibliography{biblio}

\title[Reduction of forced discrete mechanical systems]{Lagrangian
  reduction of\\ forced discrete mechanical systems}

\author[M. I. Caruso, J. Fern\'andez, C. Tori and M. Zuccalli]{}

\subjclass{Primary: 37J06, 70H33; Secondary: 70G75.}
\keywords{Geometric mechanics, forced discrete mechanical systems, symmetry and reduction.}

\begin{document}
\maketitle

\centerline{\scshape
Mat\'ias I. Caruso$^{{\href{mailto:mcaruso@mate.unlp.edu.ar}{\textrm{\Letter}}}*1,2,3}$,
Javier Fern\'andez$^{4}$,
Cora Tori$^{2,5}$
and Marcela Zuccalli$^{1,2}$}

\medskip

{\footnotesize \centerline{$^1$Depto. de Matem\'atica, Facultad de
    Ciencias Exactas, Universidad Nacional de La Plata, Argentina} }

\medskip

{\footnotesize \centerline{$^2$Centro de Matem\'atica de La Plata
    (CMaLP), Argentina} }

\medskip

{\footnotesize \centerline{$^3$Consejo Nacional de Investigaciones
    Cient\'ificas y T\'ecnicas (CONICET), Argentina} }

\medskip

{\footnotesize
\centerline{$^4$Instituto Balseiro, Universidad Nacional de Cuyo -- C.N.E.A., Argentina}
}

\medskip

{\footnotesize \centerline{$^5$Depto. de Ciencias B\'asicas, Facultad
    de Ingenier\'ia, Universidad Nacional de La Plata, Argentina} }

\bigskip

\begin{abstract}
  In this paper we propose a process of Lagrangian reduction and
  reconstruction for symmetric discrete-time mechanical systems acted
  on by external forces, where the symmetry group action on the
  configuration manifold turns it into a principal bundle. We analyze
  the evolution of momentum maps and Poisson structures under
  different conditions.
\end{abstract}


\section{Introduction}

Systems that are acted on by external forces are most common in the
modeling of real world physical systems and, in particular, mechanical
ones. The forcing may originate in actuators, friction, dissipation,
etc. As in the case of Lagrangian systems, the need for efficient
integration of the corresponding equations of motion has led to the
study of forced discrete mechanical systems. These are discrete-time
dynamical systems whose trajectories are determined by a variational
principle that is the discrete-time analogue of the one that
determines the (continuous-time) trajectories of forced mechanical
systems. Those discrete-time trajectories can be used as numerical
integrators of the (continuous) system; in the free, that is,
``unforced'' case, these are called \emph{variational integrators} and
are known to enjoy some very good conservation properties. Forced
discrete mechanical systems have been known and studied for a number
of years (see Part Three of~\cite{M-West}). Recently, the interest in
these systems has been growing (see, for
instance,~\cite{MdD-MdA-18,MdD-MdA-20,F-G-G,dL-L-LG-22}).

Naturally, when a dynamical system has some symmetries, it is useful
to look for a ``simpler'' system whose dynamics captures, at least,
part of the dynamics of the original system and may be used to
reconstruct the original one. We call such a system a \emph{reduced
  system} (references with and without constraints in the continuous
setting include \cite{C-M-R,C-M-R-01b}, while in the discrete one
include \cite{F-T-Z-16,F-T-Z}). The goal of this paper is to show
that, under adequate conditions, such systems exist for forced
discrete mechanical systems and study some of their properties. In
particular, we recover the dynamics of the original system out of that
of the reduced one. Another point that we explore is the existence of
some conserved structures in the reduced system: for example, we prove
that, if the flow of the original system preserves a Poisson structure
and the symmetry group acts by Poisson maps, then the flow of the
reduced system also preserves an induced Poisson structure. The
approach to the problem runs parallel to the one used in~\cite{F-T-Z}
for nonholonomic discrete mechanical systems, but now considering no
constraints and, instead, with the external forcing.

Other work in the study of symmetries of forced discrete mechanical
system includes \cite{M-West} and \cite{dL-L-LG-22}, where
Noether-type theorems are studied, but they do not go as far as to
present a notion of reduced system.

The plan for the paper is as follows: in Section 2 the basic notions
of forced discrete mechanical systems are reviewed. In Section 3 we
introduce the notion of symmetry group for those discrete systems and
prove the main reduction and reconstruction results. Last, in Section
4 we recall a well known Noether Theorem for forced discrete
mechanical systems and explore the evolution of momentum maps under other
conditions outside of the hypotheses of that theorem; we also study
the existence of Poisson structures conserved by the flow of the
reduced system.


\section{Forced discrete Lagrangian systems}

We dedicate this section to reviewing the definitions concerning
forced discrete mechanical systems and their dynamics.


\subsection{Preliminaries on product manifolds}
\label{subsec-preliminaries}

Given an $n$-dimensional manifold, we consider the product manifold
$Q \times Q$ and $\pr_1 : Q \times Q \lra Q$ and
$\pr_2 : Q \times Q \lra Q$ the canonical projections on the first and
second factor, respectively.  Using the product structure of
$Q \times Q$, we have that
$$T(Q \times Q) \simeq \pr_{1}^{\ast} (TQ) \oplus \pr_{2}^{\ast}
(TQ),$$ where $\pr_{i}^{\ast} (TQ)$ denotes the pullback of the
tangent bundle $TQ \lra Q$ over $Q \times Q$ by $\pr_i$ for
$i=1,2$. If we define $j_1 : \pr_{1}^{\ast} (TQ) \lra T(Q \times Q)$
as $j_1(\delta q) := (\delta q,0)$ we have that $j_1$ is an
isomorphism of vector bundles between $\pr_{1}^{\ast} (TQ)$ and
$TQ^{-} := ker(T \pr_2) \subset T(Q \times Q)$. Similarly, defining
$j_2 : \pr_{2}^{*} (TQ) \lra T(Q \times Q)$ as
$j_2(\delta q) := (0,\delta q)$ identifies $\pr_{2}^{*} (TQ)$ with the
subbundle $TQ^{+} := ker(T \pr_1) \subset T(Q \times Q)$.

So, the decomposition $T(Q \times Q) = TQ^{-} \oplus TQ^{+}$ leads to
the decomposition
$$T^{\ast}(Q \times Q) = (TQ^{-})^{\circ} \oplus (TQ^{+})^{\circ}$$
and the natural identifications
$$(TQ^{+})^{\circ} \simeq (TQ^{-})^{*} \simeq \pr_1^{*}T^{*}Q \text{ and } (TQ^{-})^{\circ} \simeq (TQ^{+})^{*} \simeq \pr_2^{*}T^{*}Q.$$

For any smooth map $H : Q \times Q \lra X$, where $X$ is a smooth
manifold, we define $D_{1}H := TH \circ j_1$ and
$D_{2}H := TH \circ j_2$, where $TH:T(Q \times Q) \lra TX$ denotes the
tangent map of $H$, as usual. Thus,
$$TH(q_0,q_1)(\delta q_0,\delta q_1) = D_{1}H(q_0,q_1)(\delta q_0) + D_{2}H(q_0,q_1)(\delta q_1).$$

If $H : Q \times Q \lra \rr$, then $D_{1}H(q_0,q_1) \in T^{*}_{q_0} Q$
and $D_{2}H(q_0,q_1) \in T^{*}_{q_1}Q$.

Naturally, these ideas may be extended to a product of more than two
(possibly different) manifolds.


\subsection{Main ingredients}

Recall that a forced Lagrangian mechanical system consists of a triple
$(Q,L,f)$, where $Q$ is a smooth manifold, the \emph{configuration
  space}, $L : TQ \lra \rr$ is a smooth map, the \emph{Lagrangian},
and $f$ is a horizontal\footnote{A differential form on the total
  space of a fiber bundle $\phi : Q \lra M$ is called {\it horizontal}
  if it vanishes when any of its arguments is a vertical vector, i.e.,
  an element of $\ker T\phi$.} $1$-form on $TQ$, the
\emph{force}\footnote{Equivalently, it is common to define the force
  of a system as a fiber-preserving map from $TQ$ to $T^*Q$.}.

In a discrete-time analogue of this type of system, it seems
reasonable to replace the infinitesimal processes with states of the
system corresponding to close moments in time, i.e., to replace the
tangent bundle $TQ$ with the product manifold $Q \times Q$ (see, for
example, \cite{M-West} and \cite{F-G-G}):

\begin{definition}
  A {\it forced discrete mechanical system} (FDMS) consists of a
  triple $(Q,L_d,f_d)$ where $Q$ is an $n$-dimensional differential
  manifold, the {\it configuration space}, $L_d : Q \times Q \lra \rr$
  is a smooth map, the {\it discrete Lagrangian} and $f_d$ is a
  $1$-form on $Q \times Q$, the {\it discrete force}.
\end{definition}

Using the identifications stated in Section
\ref{subsec-preliminaries}, we will usually decompose forces
$f_d \in \Gamma(Q \times Q, T^{*}(Q\times Q))$ as
$f_d^{-} \oplus f_d^{+}$ with $f_d^{-} \in TQ^-$ and
$f_d^{+} \in TQ^+$, where $\Gamma(M,E)$ denotes the space of smooth
sections of a fiber bundle $E \lra M$.

Thus, $f_d^- (q_0,q_1) \in T_{q_0}^*Q$,
$f_d^+ (q_0,q_1) \in T_{q_1}^*Q$ and
$$f_d(q_0,q_1)(\delta q_0,\delta q_1) = f_d^- (q_0,q_1)(\delta q_0) + f_d^+ (q_0,q_1)(\delta q_1).$$

\begin{remark}
  Roughly speaking, a discretization of $TQ$ is a local diffeomorphism
  $\Delta : TQ \lra Q \times Q$, mapping the zero section onto the
  diagonal. There are several specializations of this basic idea, as
  the discretization maps of \cite{BL-MdD}.
	
  Given a forced mechanical system $(Q,L,f)$, it is possible to
  construct a FDMS $(Q,L_d,f_d)$ via a discretization map
  $\Delta : TQ \lra Q \times Q$ as follows: the discrete Lagrangian is
  $L_d := (\Delta^{-1})^* L = L \circ \Delta^{-1}$ and the discrete
  force is $f_d := (\Delta^{-1})^* f$, where $(\Delta^{-1})^*$ denotes
  the pullback by $\Delta^{-1}$. In practice, however, it is common to
  have a family of discretizations $\Delta_h$ depending on a parameter
  $h > 0$. In this case, the definitions of $L_d$ and $f_d$ are
  modified multiplying by $h$.
\end{remark}

\begin{example}\label{example-discretization}
  Let $(Q,L,f)$ be the forced mechanical system considered in Example
  2 of \cite{dL-L-LG-22}, where $Q = \rr^2$ with $q = (x,y)$,
  $$L(q,\dot{q}) = \frac{1}{2} \norm{\dot{q}}^2 - \norm{q}^2 \left( \norm{q}^2 - 1 \right)^2$$
  and
  $$f(x,y,\dot{x},\dot{y}) = - k (\dot{x} dx + \dot{y} dy).$$
	
  Considering the midpoint discretization, defined for a parameter $h > 0$ as
  $$\Delta_h(q,\dot{q}) := \left( q - \frac{h}{2} \dot{q} , q + \frac{h}{2} \dot{q} \right),$$
  with inverse
  $$\Delta_h^{-1}(q_0,q_1) := \left( \frac{q_0 + q_1}{2} , \frac{q_1 - q_0}{h} \right),$$
  we obtain the FDMS $(\rr^2,L_d,f_d)$ given by
  \[
    \begin{split}
      L_d&(x_0,y_0,x_1,y_1) = \\
      & \frac{h}{2} \left[ \left( \frac{x_1 - x_0}{h} \right)^2 + \left( \frac{y_1 - y_0}{h} \right)^2 \right] \\
      & - h \left[ \left( \frac{x_0 + x_1}{2} \right)^2 + \left(
          \frac{y_0 + y_1}{2} \right)^2 \right] \left( \left[ \left(
            \frac{x_0 + x_1}{2} \right)^2 + \left( \frac{y_0 + y_1}{2}
          \right)^2 \right] - 1 \right)
    \end{split}
  \]
  and
  \[
    \begin{split}
      f_d&(x_0,y_0,x_1,y_1) = \\
      & - \frac{k}{2} \left[ (x_1 - x_0) \ dx_0 + (y_1 - y_0) \ dy_0 + (x_1 - x_0) \ dx_1 + (y_1 - y_0) \ dy_1 \right].
    \end{split}
  \]
\end{example}


\subsection{Dynamics}

The trajectories of a forced Lagrangian mechanical system are
determined by a variational principle (see, for example, Section 3.1
of \cite{M-West}). Next, we review a similar notion for FDMS.

\begin{definition}
  A {\it discrete curve} in $Q$ is a map
  $q_\cdot : \{ 0,\ldots,N \} \lra Q$ and an {\it infinitesimal variation}
  over a discrete curve $q_\cdot$ consists of a map
  $\delta q_\cdot : \{ 0,\ldots,N \} \lra TQ$ such that
  $\delta q_k \in T_{q_k}Q, \ \forall\, k=0,...,N $. An infinitesimal
  variation is said to have {\it fixed endpoints} if $\delta q_0 = 0$
  and $\delta q_N = 0$.
\end{definition}

\begin{definition}
  The {\it discrete action functional} of the FDMS $(Q,L_d,f_d)$ is
  defined as
  $$\frakS_d(q_\cdot) := \sum_{k=0}^{N-1} L_d(q_k,q_{k+1}).$$
\end{definition}

The dynamics of a FDMS is given by the appropriately modified discrete
Hamilton principle known as discrete Lagrange--d'Alembert
principle. See, e.g., Section 3.2. of \cite{M-West}.

\begin{definition}
  A discrete curve $q_\cdot$ is a {\it trajectory} of the FDMS
  $(Q,L_d,f_d)$ if it satisfies
  \begin{equation}\label{forced-variational principle}
    \delta \left( \sum_{k=0}^{N-1} L_d(q_{k},q_{k+1})  \right) + \sum_{k=0}^{N-1} f_d(q_k,q_{k+1})(\delta q_k,\delta q_{k+1}) = 0,
  \end{equation}
  for all infinitesimal variations $\delta q_\cdot$ of $q_\cdot$ with fixed endpoints. 
\end{definition}

The following well known result (see \cite{M-West}), which follows
from the standard calculus of variations, characterizes the
trajectories of the system in terms of solutions of a set of algebraic
equations.

\begin{theorem}
  Let $(Q,L_d,f_d)$ be a FDMS. Then, a discrete curve
  $q_\cdot : \{ 0,\ldots,N \} \lra Q$ is a trajectory of $(Q,L_d,f_d)$ if
  and only if it satisfies the following algebraic identities
\begin{equation}\label{forcedELe}
  D_2 L_d(q_{k-1},q_k) + D_1 L_d(q_k,q_{k+1}) + f_d^+(q_{k-1},q_k) + f_d^-(q_k,q_{k+1}) = 0 \in T_{q_k}^{*}Q
\end{equation}
for all $k = 1,\ldots,N-1$, called the {\it forced discrete
  Euler-Lagrange equations}.
\end{theorem}


\section{Reduction of forced discrete Lagrangian systems}

In this section we turn our attention to forced discrete Lagrangian
systems with symmetries. We define the notion of symmetry group of a
FDMS and prove that symmetric systems can be reduced. Finally, we
prove that the dynamics of the original system may be reconstructed in
terms of that of the reduced one.


\subsection{Principal and discrete connections}

We begin this section recalling that given a left action $l^M$ of a
Lie group $G$ on a manifold $M$, the {\it infinitesimal generator} of
the action corresponding to an element $\xi \in \frakg := \Lie(G)$ is
the vector field on $M$ given by
\begin{equation}\label{inf-generator}
  \xi_M(x) := \frac{d}{dt} \bigg|_{t=0} l_{\exp(t \xi)}^M(x).
\end{equation}

In the continuous setting, where the dynamics of the system takes
place on the tangent bundle, if there is a symmetry group for the
system, the following notion may be used to define a reduction
procedure, as shown in \cite{C-M-R}.

\begin{definition}
  A {\it principal connection} on a principal $G$-bundle
  $\pi : M \lra M/G$ is a differential $1$-form
  $\frakA : TM \lra \frakg$ such that
  \begin{enumerate}
  \item $\frakA(\xi_M(x)) = \xi$ for all $\xi \in \frakg$ and $x \in M$,
  \item $\frakA(Tl_g^M (v)) = \Ad_g(\frakA(v))$ for all $v \in TM$,
    where $\Ad$ is the adjoint action of $G$ on $\frakg$.
  \end{enumerate}
\end{definition}

In \cite{C-M-R}, Section 2.4., these objects are used for constructing
a model for $TM/G$ that is easier to handle. Given a principal
connection $\frakA$ on a principal $G$-bundle $\pi : M \lra M/G$,
there is an isomorphism of vector bundles over $M/G$,
$\alpha_\frakA : TM/G \lra T(M/G) \oplus \tilde{\frakg}$ given by
$$\alpha_\frakA([v_x]) := T\pi (v_x) \oplus [x,\frakA(v_x)],$$
where $\tilde{\frakg} := (M \times \frakg)/G$ is the {\it adjoint
  bundle}, with $G$ acting on $M$ by the action defining the principal
$G$-bundle and on $\frakg$ by the adjoint action. $\tilde{\frakg}$ is
a vector bundle over $M/G$ with fiber $\frakg$ and projection map
induced by $\pr_1 : M \times \frakg \lra M$.

Its inverse
$\alpha_\frakA^{-1} : T(M/G) \oplus \tilde{\frakg} \lra TM/G$ is given
by
$$\alpha_\frakA^{-1}(u \oplus [x,\xi]) := [h^x(u) + \xi_M(x)],$$
where $h^x(u)$ is the horizontal lift associated to $\frakA$ of $u$ at
$x \in M$.

The dual map induces an isomorphism of vector bundles
$(\alpha_\frakA^{-1})^* : T^*M/G \lra T^*(M/G) \oplus
\tilde{\frakg}^*$ by
$$(\alpha_\frakA^{-1})^*([\alpha_x])(u \oplus [x,\xi]) = [\alpha_x](\alpha_\frakA^{-1}(u \oplus [x,\xi])) = [\alpha_x]([h^x(u) + \xi_M(x)]),$$
for $u \oplus [x,\xi] \in T(M/G) \oplus \tilde{\frakg}$.

Consider a principal $G$-bundle $\pi : Q \lra Q/G$, with $l^Q$ the
action of $G$ on $Q$ defining it and let $l^{Q \times Q}$ be the
diagonal action of $G$ on $Q \times Q$. The analogous idea in the
discrete setting, where we replace $TQ$ with $Q \times Q$, would be to
describe the quotient $(Q \times Q)/G$, using some other tool, since
$Q \times Q$ is not a tangent bundle.

In \cite{F-T-Z}, inspired by the ideas found in \cite{L-M-W}, the
authors construct a model space for $(Q\times Q) /G$ associated to a
geometric object on $\pi:Q\rightarrow Q/G$ called affine discrete
connection. Briefly, a discrete connection on $\pi:Q\rightarrow Q/G$
consists in choosing a submanifold of $Q \times Q$ with certain
charactristics. Alternatively, this object can be described as a
function on $Q\times Q$ with values in $G$ that satisfies certain
properties.  Actually, a discrete connection is defined on an open
$\mathcal{U} \subset Q \times Q$ called domain of the discrete
connection but for simplicity in this paper we will consider that all
discrete connections are defined on $Q\times Q$.

\begin{definition}
  Let $\gamma : Q \lra G$ be a smooth $G$-equivariant map with respect
  to $l^Q$ and $l^G$, where $l^G$ is the action of $G$ on itself by
  conjugation. An {\it affine discrete connection} $\calA_d$ with {\it
    level} $\gamma$ is a smooth map $\calA_d : Q \times Q \lra G$
  satisfying
  \begin{enumerate}
  \item For all $q_0,q_1 \in Q$, $g_0,g_1 \in G$,
    $$\calA_d(l^Q_{g_0}(q_0) , l^Q_{g_1}(q_1)) = g_1 \calA_d(q_0,q_1) g_0^{-1}.$$
  \item $\calA_d(q,l^Q_{\gamma(q)}(q)) = e$, where $e$ is the identity
    of $G$.
  \end{enumerate}
\end{definition}

Following the ideas found in Section 4.2. of \cite{F-T-Z}, just as a
principal connection allows us to identify the quotient $(TM)/G$ with
a different model, a discrete connection $\calA_d$ provides an
isomorphism of fiber bundles over $Q/G$,
$\Phi_{\calA_d} : (Q \times Q)/G \lra \tilde{G} \times Q/G$, where
$\tilde{G} := (Q \times G)/G$ is the {\it conjugate bundle} (in which
$G$ acts on $Q$ and $G$ with $l^Q$ and $l^G$, respectively) and
$\Phi_{\calA_d}$ is defined dropping to the quotient the
$G$-equivariant map
$\tilde{\Phi}_{\calA_d} : Q \times Q \lra Q \times G \times (Q/G)$
given by
$$\tilde{\Phi}_{\calA_d}(q_0,q_1) := (q_0,\calA_d(q_0,q_1),\pi(q_1)).$$

Its inverse
$\tilde{\Psi}_{\calA_d} : Q \times G \times (Q/G) \lra Q \times Q$,
which we will use in the following pages, is of the form
$$\tilde{\Psi}_{\calA_d}(q_0,w_0,\tau_1) = (q_0,\tilde{F}_1(q_0,w_0,\tau_1)),$$
for a function $\tilde{F}_1 : Q \times G \times (Q/G) \lra G$ whose
explicit definition may be found in Section 4 of \cite{F-T-Z}.

If we call $\rho : Q \times G \lra \tilde{G}$ and
$\tilde{\pi} : Q \times Q \lra (Q \times Q)/G$ the quotient maps and
define $\Upsilon : Q \times Q \lra \tilde{G} \times Q/G$ by
$\Phi_{\calA_d} \circ \tilde{\pi}$, we have the following commutative
diagram:
\begin{equation}\label{red-spaces}
  \xymatrixcolsep{5pc}\xymatrixrowsep{4pc}\xymatrix{
    Q \times Q \ar[d]_{\tilde{\pi}}
    \ar[r]^{\tilde{\Phi}_{\calA_d}} \ar[dr]_{\Upsilon} & Q \times G \times (Q/G) \ar[d]^{\rho \times 1_{Q/G}} \\
    (Q\times Q)/G \ar[r]_{\Phi_{\calA_d}} & \tilde{G} \times Q/G
  }
\end{equation}

We summarize the two isomorphisms we will use below in a single statement.

\begin{theorem}
  Let $\calA_d$ be an affine discrete connection on $\pi : Q \lra Q/G$
  and $\tilde{\frakA}$ be a principal connection on
  $\tilde{\pi} : Q \times Q \lra (Q \times Q)/G$. There exist
  isomorphisms of bundles in the corresponding categories
$$T(Q \times Q)/G \simeq T \left( (Q \times Q)/G \right) \oplus \tilde{\frakg}, \quad (Q \times Q)/G \simeq \tilde{G} \times Q/G.$$
\end{theorem}

\begin{remark}\label{Upsilon-bundle}
  Since $\tilde{\pi}$ is a principal $G$-bundle and $\Phi_{\calA_d}$
  is a diffeomorphism, it is clear that
  $\Upsilon : Q \times Q \lra \tilde{G} \times Q/G$ defines a
  principal $G$-bundle.
\end{remark}


\subsection{Forced discrete Lagrangian systems with symmetries}

\begin{definition}\label{sym-group-def}
  Let $(Q,L_d,f_d)$ be a FDMS and let $G$ be a Lie group acting on $Q$
  on the left so that the quotient map $\pi : Q \lra Q/G$ is a
  principal $G$-bundle. Considering the diagonal action on
  $Q \times Q$, we say that $G$ is a {\it symmetry group} of
  $(Q,L_d,f_d)$ if $L_d$ is $G$-invariant and
  $f_d : Q \times Q \lra T^*(Q \times Q)$ is $G$-equivariant,
  considering on $T^*(Q \times Q)$ the cotangent lift
  $l^{T^*(Q \times Q)}$ of the diagonal action\footnote{The
    $G$-equivariance of $f_d$ as a section is equivalent to its
    $G$-invariance as a $1$-form.}.
\end{definition}

In \cite{F-T-Z}, the authors introduce a reduction procedure for
symmetric nonholonomic discrete Lagrangian systems. If we consider the
unconstrained case, their work shows that a discrete mechanical system
$(Q,L_d)$ with symmetry group $G$ induces a ``reduced'' dynamical
system on $\tilde{G} \times Q/G$. Our goal here is to adapt their
formalism so that it can handle forces. Therefore, we need to study
how the forces $f_d$ drop to $\tilde{G} \times Q/G$.

The diagram \eqref{red-spaces} shows the reduced spaces and, if $G$ is
a symmetry group of the system $(Q,L_d,f_d)$, we can add the reduced
Lagrangians to the picture to obtain

\begin{equation*}
  \xymatrixcolsep{5pc}\xymatrixrowsep{4pc}\xymatrix{
    & & \rr \\ 
    Q \times Q \ar@/^/[urr]^{L_d} \ar[d]_{\tilde{\pi}}
    \ar[r]^{\tilde{\Phi}_{\calA_d}} \ar[dr]_{\Upsilon} & Q \times G \times (Q/G) \ar[ur]^(.4){\check{L}_d} \ar[d]^{\rho \times 1_{Q/G}} & \\
    (Q\times Q)/G \ar[r]_{\Phi_{\calA_d}} & \tilde{G} \times Q/G \ar@/_1.5pc/[uur]_{\hat{L}_d} &
  }
\end{equation*}

Therefore, the only missing piece is the force $f_d$ and the
corresponding variational principle. Since the dynamics of the reduced
system will take place on $\tilde{G} \times Q/G$, which is our model
space for $(Q \times Q)/G$, we want $f_d$ to ``drop'' to a section of
a bundle involving the cotangent bundle $T^*(\tilde{G} \times Q/G)$.

Setting $\Psi_{\calA_d} := \Phi_{\calA_d}^{-1}$, we can consider the
following diagram:
$$\xymatrixcolsep{5pc}\xymatrixrowsep{4pc}\xymatrix{
  \Psi_{\calA_d}^*(T^*(Q \times Q)/G) \ar[d]_{\pr_1}
  \ar[r]^{\pr_2} & T^*(Q \times Q)/G \ar[d]^{\tilde{\pi}_G} \\
  \tilde{G} \times Q/G \ar[r]_{\Psi_{\calA_d}} & (Q \times Q)/G }$$
where $\tilde{\pi}_G([\alpha_{(q_0,q_1)}]) := [(q_0,q_1)]$ and
\begin{equation*}
  \begin{split}
    \Psi_{\calA_d}^* (T^*(Q \times Q)/G) &= \{ ((v_0,\tau_1) , [\alpha_{(q_0,q_1)}] ) \in (\tilde{G} \times Q/G) \times (T^*(Q \times Q)/G) \\
    &\mid \Psi_{\calA_d}(v_0,\tau_1) = [(q_0,q_1)] \}.
  \end{split}
\end{equation*}

We want to identify the space $\Psi_{\calA_d}^*(T^*(Q \times Q)/G)$
with another vector bundle over $\tilde{G} \times Q/G$ that has a
simpler structure. In order to do that, we focus on the tangent side
first.

\begin{proposition}\label{isomorphisms-prop}
  Let $\calA_d$ be an affine discrete connection on $\pi : Q \lra Q/G$
  and $\tilde{\frakA}$ be a principal connection on
  $\tilde{\pi} : Q \times Q \lra (Q \times Q)/G$. There exists an
  isomorphism of vector bundles over $\tilde{G} \times Q/G$
  $$\Psi_{\calA_d}^*(T(Q \times Q)/G) \simeq T(\tilde{G} \times Q/G) \oplus \hat{\frakg},$$
  where $\hat{\frakg} := \Psi_{\calA_d}^* \tilde{\frakg}$.
  
  Explicitly, the isomorphism is given by
  \begin{equation*}
    \begin{split}
      \calX : \Psi_{\calA_d}^*(T(Q \times Q)/G) &\lra T(\tilde{G} \times Q/G) \oplus \hat{\frakg} \\
      \calX((v_0,\tau_1) , [\delta q_0,\delta q_1]) &:=
      T\Upsilon(\delta q_0,\delta q_1) \oplus [(q_0,q_1) ,
      \tilde{\frakA}(\delta q_0,\delta q_1)],
    \end{split}
  \end{equation*}
  where $\Psi_{\calA_d}(v_0,\tau_1) = [(q_0,q_1)]$ and we are omitting
  the first factor of the element of $\hat{\frakg}$, which is
  $(v_0,\tau_1$). Its inverse is given by
  \begin{equation*}
    \begin{split}
      \calY : T(\tilde{G} \times Q/G) \oplus \tilde{\frakg} &\lra \Psi_{\calA_d}^*(T(Q \times Q)/G) \\
      \calY((\delta v_0,\delta \tau_1) \oplus [(q_0,q_1),\xi]) &:= \left( (v_0,\tau_1) , \left[ h^{(q_0,q_1)} \left( T\Psi_{\calA_d}(\delta v_0,\delta \tau_1) \right) + \xi_{Q \times Q}(q_0,q_1) \right] \right),
    \end{split}
  \end{equation*}
  where $(v_0,\tau_1) = \Upsilon(q_0,q_1)$.
\end{proposition}

We will use the following auxiliary lemma to prove the proposition.

\begin{lemma}\label{pullback-difeo}
  Let $\chi : W \lra X$ be a vector bundle and $Y$ be a smooth
  manifold. If $\Phi : Y \lra X$ is a diffeomorphism, then there
  exists an isomorphism of vector bundles $\Phi^* W \simeq W$.
\end{lemma}

\begin{proof}
  We have the diagram
  $$\xymatrixcolsep{5pc}\xymatrixrowsep{4pc}\xymatrix{
    \Phi^* W \ar[d]_-{\pr_1^{\Phi^* W}} \ar[r]^-{\pr_2^{\Phi^* W}} & W \ar[d]^-\chi \\
    Y \ar[r]_-\Phi & X \\
  }$$ where
  $$\Phi^* W = \{ (y,w_x) \mid \Phi(y) = \chi(w_x) \} \subs Y \times W$$
  and the projections $\pr_i^{\Phi^*W}$ are the usual ones on
  each factor. It is easy to check that $\pr_2^{\Phi^*W}$ is the
  required isomorphism.
\end{proof}

\begin{proof}[Proof of Proposition \ref{isomorphisms-prop}]
  Let us apply Lemma \ref{pullback-difeo} to our context of the
  proposition. If $\pi : Q \lra Q/G$ is a principal $G$-bundle, given
  an affine discrete connection $\calA_d$, we have the following
  diagram:
  $$\xymatrixcolsep{7pc}\xymatrixrowsep{5pc}\xymatrix{
    \Psi_{\calA_d}^*(T(Q \times Q)/G) \ar[d]_-{\pr_1^{\Psi_{\calA_d}^*(T(Q \times Q)/G)}} \ar[r]^-{\pr_2^{\Psi_{\calA_d}^*(T(Q \times Q)/G)}} & T(Q \times Q)/G \ar[d] \\
    \tilde{G} \times Q/G \ar[r]_-{\Psi_{\calA_d}} & (Q \times Q)/G \\
  }$$
	
  Since $\Psi_{\calA_d}$ is a diffeomorphism, by Lemma \ref{pullback-difeo}, there exists an isomorphism of vector bundles
  \begin{equation}\label{isomorphisms-prop-1}
    \Psi_{\calA_d}^*(T(Q \times Q)/G) \simeq \frac{T(Q \times Q)}{G}.
  \end{equation}
  
  On the other hand, since $T\Psi_{\calA_d}$ and
  $\pr_2^{\Psi_{\calA_d}^* \tilde{\frakg}}$ are both isomorphisms of
  vector bundles (the former because $\Psi_{\calA_d}$ is a
  diffeomorphism and the latter because of Lemma \ref{pullback-difeo},
  applied to the adjoint bundle $\tilde{\frakg}$), we can define an
  isomorphism of vector bundles
  \[
    \begin{split}
      \sigma : T(\tilde{G} \times Q/G) \oplus \Psi_{\calA_d}^* \tilde{\frakg} &\lra T((Q \times Q)/G) \oplus \tilde{\frakg} \\
      \sigma &:= T\Psi_{\calA_d} \oplus \pr_2^{\Psi_{\calA_d}^*
        \tilde{\frakg}}.
    \end{split}
  \]
	
  Therefore, given a principal connection $\tilde{\frakA}$ on
  $\tilde{\pi} : Q \times Q \lra (Q \times Q)/G$, we have the
  isomorphisms
  $$\xymatrixcolsep{5pc}\xymatrixrowsep{4pc}\xymatrix{
    T(Q \times Q)/G \ar[r]^-{\alpha_{\tilde{\frakA}}} \ar[dr] & T((Q \times Q)/G) \oplus \tilde{\frakg} \ar[d] & T(\tilde{G} \times Q/G) \oplus \Psi_{\calA_d}^* \tilde{\frakg} \ar[l]_-{\sigma} \ar[d] \\
    & (Q \times Q)/G & \tilde{G} \times Q/G \ar[l]_-{\Psi_{\calA_d}}
  }$$
	
  Adding this to the isomorphism \eqref{isomorphisms-prop-1}, we have
  an isomorphism of vector bundles over $\tilde{G} \times Q/G$
  $$\Psi_{\calA_d}^*(T(Q \times Q)/G) \simeq T(\tilde{G} \times Q/G) \oplus \Psi_{\calA_d}^* \tilde{\frakg},$$
  where
  \begin{equation*}
    \begin{split}
      \Psi_{\calA_d}^* \tilde{\frakg} &= \{ \left( (v_0,\tau_1) , [(q_0,q_1),\xi] \right) \mid \Psi_{\calA_d}(v_0,\tau_1) = [(q_0,q_1)] \} \\
      &= \{ \left( (v_0,\tau_1) , [(q_0,q_1),\xi] \right) \mid
      (v_0,\tau_1) = \Upsilon(q_0,q_1) \}.
    \end{split}
  \end{equation*}
	
  With this idea in mind, to avoid a clumsier notation, we have called
  this fiber bundle $\hat{\frakg}$.
\end{proof}

\begin{corollary}
  Let $\calA_d$ be an affine discrete connection on $\pi : Q \lra Q/G$
  and $\tilde{\frakA}$ be a principal connection on
  $\tilde{\pi} : Q \times Q \lra (Q \times Q)/G$. There exists an
  isomorphism of vector bundles over $\tilde{G} \times Q/G$
$$\Psi_{\calA_d}^*(T^*(Q \times Q)/G) \simeq T^*(\tilde{G} \times Q/G) \oplus \hat{\frakg}^*.$$
\end{corollary}

Considering the isomorphisms just described, we have the following result.

\begin{proposition}\label{f_d-hat}
  Let $\calA_d$ be an affine discrete connection on $\pi : Q \lra Q/G$
  and $\tilde{\frakA}$ be a principal connection on
  $\tilde{\pi} : Q \times Q \lra (Q \times Q)/G$. A $G$-equivariant
  section $f_d : Q \times Q \lra T^*(Q \times Q)$ induces a section
  $\hat{f}_d : \tilde{G} \times (Q/G) \lra T^*(\tilde{G} \times (Q/G))
  \oplus \hat{\frakg}^*$. Explicitly,
  $$\hat{f}_d(v_0,\tau_1)((\delta v_0,\delta \tau_1) \oplus [(q_0,q_1),\xi]) = f_d(q_0,q_1)(\delta q_0,\delta q_1),$$
  for $(\delta q_0,\delta q_1) \in T(Q \times Q)$ over $(q_0,q_1)$
  such that
  $(\delta v_0,\delta \tau_1) = T\Upsilon (\delta q_0,\delta q_1)$ and
  $\xi = \tilde{\frakA}(\delta q_0,\delta q_1)$.
\end{proposition}


\subsection{Reduced dynamics}

In what follows, $\pi : Q \lra Q/G$ will be a principal $G$-bundle and
$\calA_d$ and $\frakA$ will be an affine discrete connection and a
principal connection on it, respectively.

The following lemma, whose proof is a straightforward computation,
will allow us to shorten the statement of the theorem we want to
state.

\begin{lemma}
  A principal connection $\frakA$ on $\pi : Q \lra Q/G$ induces a
  principal connection $\tilde{\frakA}$ on
  $\tilde{\pi} : Q \times Q \lra (Q \times Q)/G$ given by
  $\tilde{\frakA} := \frac{1}{2} (\pr_1^* \frakA + \pr_2^*
  \frakA)$. Explicitly,
  $$\tilde{\frakA}(\delta q_0,\delta q_1) := \frac{1}{2} \left( \frakA(\delta q_0) + \frakA(\delta q_1) \right).$$
\end{lemma}

Let $(Q,L_d,f_d)$ be a FDMS. We will use the following version of the
force $f_d$:
$\check{f}_d := \tilde{\Psi}_{\calA_d}^* f_d \in \Omega^1(Q \times G
\times Q/G)$, which can be decomposed as
\begin{equation*}
  \begin{split}
    \check{f}_d(q_0,w_0,\tau_1)(\delta q_0,\delta w_0,\delta \tau_1) &= \check{f}_d^1(q_0,w_0,\tau_1)(\delta q_0) + \check{f}_d^2(q_0,w_0,\tau_1)(\delta w_0) \\
    & \ \ \ \ \ + \check{f}_d^3(q_0,w_0,\tau_1)(\delta \tau_1)
  \end{split}
\end{equation*}
for all
$(\delta q_0,\delta w_0,\delta \tau_1) \in T_{(q_0,w_0,\tau_1)}(Q
\times G \times Q/G)$.

\begin{theorem}\label{teor-red-forzada}
  Let $q_\cdot = (q_0,\ldots,q_N)$ be a discrete curve in $Q$ and let
  \begin{equation*}
    \begin{split}
      \tau_k &:= \pi(q_k), \quad 1 \le k \le N, \\
      w_k &:= \calA_d(q_k,q_{k+1}), \quad v_k := \rho(q_k,w_k), \quad
      0 \le k \le N-1
    \end{split}
  \end{equation*}
  be the corresponding discrete curves in $Q/G$, $G$ and
  $\tilde{G}$. Then, given a forced discrete mechanical system
  $(Q,L_d,f_d)$ with symmetry group $G$, the following statements are
  equivalent.
  \begin{enumerate}
  \item \label{it:teor-red-forzada-vp_q} $q_\cdot$ satisfies the
    variational principle \eqref{forced-variational principle}.
  \item \label{it:teor-red-forzada-em_q} $q_\cdot$ satisfies the forced
    discrete Euler-Lagrange equations \eqref{forcedELe}.
  \item \label{it:teor-red-forzada-rvp_q}
    $(v_\cdot,\tau_\cdot) = ((v_0,\tau_1),\ldots,(v_{N-1},\tau_N))$
    satisfies
    \begin{equation*}
      \begin{split}
        \delta &\left( \sum_{k=0}^{N-1} \hat{L}_d(v_k,\tau_{k+1}) \right) \\
        & \ \ \ \ \ + \sum_{k=0}^{N-1} \hat{f}_d(v_k,\tau_{k+1})
        \left( (\delta v_k,\delta \tau_{k+1}) \oplus \left[
            (q_k,q_{k+1}) , \frac{1}{2}(\xi_k + \xi_{k+1}) \right]
        \right) = 0
      \end{split}
    \end{equation*}
    for all infinitesimal variations
    $(\delta v_\cdot, \delta \tau_\cdot,\xi_\cdot)$ satisfying:
    \begin{enumerate}
    \item $\delta \tau_k \in T_{\tau_k} (Q/G)$ for $k=1,\ldots,N$,
    \item $\xi_\cdot =(\xi_0,\ldots,\xi_N)\in \Lie(G)^{N+1}$,
    \item for $k=0,\ldots,N-1$,
      \[
        \begin{split}
          \delta v_k &= T_{(q_k,w_k)}\rho(h^{q_k}(\delta
          \tau_k),T_{(q_k,q_{k+1})}\mathcal{A}_d(h^{q_k}(\delta
          \tau_k),h^{q_{k+1}}(\delta \tau_{k+1}))) \\
          & \ \ \ + T_{(q_k,w_k)}\rho((\xi_k)_Q(q_k),
          T_{(q_k,q_{k+1})}\mathcal{A}_d((\xi_k)_Q(q_k),
          (\xi_{k+1})_Q(q_{k+1})))
        \end{split}
      \]
      where $h$ is the horizontal lift associated to $\frakA$,
      $\delta \tau_0\in T_{\pi(q_0)}(Q/G)$,
    \item and the ``fixed endpoint conditions":
      $\delta \tau_0 = 0$\hspace{.1pc}\footnote{Notice that
        $\delta \tau_0$ is just a convenient device to simplify the
        expression of the $\delta v_k$.}, $\delta \tau_N=0$,
      $\xi_0=\xi_N=0$.
    \end{enumerate}
    
  \item \label{it:teor-red-forzada-rem_q} $(v_\cdot,\tau_\cdot)$ satisfies the
    following conditions for each fixed
    $(v_{k-1},\tau_k,v_k,\tau_{k+1})$, with $1 \le k \le N-1$.
    \begin{itemize}
    \item $\phi_{f_d} = 0$ for $\phi_{f_d} \in T_{\tau_k}^*(Q/G)$
      defined by
      \begin{equation}\label{red-traj-eq1}
        \begin{split}
          \phi_{f_d} &:= D_1 \check{L}_d (q_k,w_k,\tau_{k+1}) \circ h^{q_k} + D_3 \check{L}_d (q_{k-1},w_{k-1},\tau_k) \\
          &+ D_2 \check{L}_d (q_k,w_k,\tau_{k+1}) D_1\calA_d(q_k,q_{k+1}) \circ h^{q_k} \\
          &+ D_2 \check{L}_d (q_{k-1},w_{k-1},\tau_k) D_2\calA_d(q_{k-1},q_k) \circ h^{q_k} \\
          &+ \check{f}_d^1(q_k,w_k,\tau_{k+1}) \circ h^{q_k} + \check{f}_d^3(q_{k-1},w_{k-1},\tau_k) \\
          &+ \check{f}_d^2(q_k,w_k,\tau_{k+1}) D_1\calA_d(q_k,q_{k+1}) \circ h^{q_k} \\
          &+ \check{f}_d^2(q_{k-1},w_{k-1},\tau_k)
          D_2\calA_d(q_{k-1},q_k) \circ h^{q_k}.
        \end{split}
      \end{equation}
			
    \item $\psi_{f_d}(\xi_k) = 0$ for all $\xi_k$, where
      $\psi_{f_d} \in \frakg^*$ is defined by
      \begin{equation}\label{red-traj-eq2}
        \begin{split}
          \psi_{f_d}(\xi_k) &:= \left( D_2 \check{L}_d (q_{k-1},w_{k-1},\tau_k) w_{k-1}^{-1} - D_2 \check{L}_d (q_k,w_k,\tau_{k+1}) w_k^{-1} \right) (\xi_k) \\
          &+ \left( \check{f}_d^1(q_k,w_k,\tau_{k+1}) + \check{f}_d^2(q_k,w_k,\tau_{k+1}) D_1 \calA_d(q_k,q_{k+1}) \right. \\
          &+ \left. \check{f}_d^2(q_{k-1},w_{k-1},\tau_k) D_2 \calA_d
            (q_{k-1},q_k) \right) \left( (\xi_k)_Q(q_k) \right),
        \end{split}
      \end{equation}
      where we are using the notation $w \xi := T_e L_w (\xi)$, for
      $w \in G$ and $\xi \in \frakg$, and with $L_w$ being the left
      translation.
    \end{itemize}
  \end{enumerate}
\end{theorem}

\begin{proof}
  \leavevmode \underline{$1 \Longleftrightarrow 2$}. It is a known
  result that can be found, for example, on Part Three of
  \cite{M-West} and in Section 5 of \cite{K-M-O-W}.

  \underline{$3 \Lra 1$}. Given a vanishing end points infinitesimal
  variation $\delta q_\cdot$ on $q_\cdot$, for $0 \le k \le N$, we can define
  $\delta \tau_k := T_{q_k}\pi (\delta q_k)$ and
  $\xi_k := \frakA(\delta q_k)$, so that
  $\delta q_k = h^{q_k}(\delta \tau_k) + (\xi_k)_Q(q_k)$. In addition,
  we can define $\delta v_k$ by
  \begin{equation*}
    \begin{split}
      \delta v_k &:= T_{(q_k,w_k)} \rho \left( h^{q_k}(\delta \tau_k) , T_{(q_k,q_{k+1})} \calA_d \left( \left( h^{q_k}(\delta \tau_k) , h^{q_{k+1}}(\delta \tau_{k+1}) \right) \right) \right) \\
      &+ T_{(q_k,w_k)} \rho \left( (\xi_k)_Q(q_k) , T_{(q_k,q_{k+1})}
        \calA_d \left( \left( (\xi_k)_Q(q_k) , (\xi_{k+1})_Q(q_{k+1})
          \right) \right) \right),
    \end{split}
  \end{equation*}
  for $0 \le k \le N-1$, so that
  $$(\delta v_k,\delta \tau_{k+1}) = T_{(q_k,q_{k+1})} \Upsilon (\delta q_k,\delta q_{k+1})$$
  and the end points satisfy the conditions of the statement, because
  $\delta q_\cdot$ has vanishing end points.
	
  Following the proof of Theorem 5.11 of \cite{F-T-Z}, we have
  $$d\frakS_d(q_\cdot)(\delta q_\cdot) = d\hat{\frakS}_d(v_\cdot,\tau_\cdot)(\delta v_\cdot,\delta \tau_\cdot),$$
  where
  $$\hat{\frakS}_d(v_\cdot,\tau_\cdot) := \sum_{k=0}^{N-1} \hat{L}_d(v_k,\tau_{k+1}).$$
	
  On the other hand, recalling Proposition \ref{f_d-hat},
  $$\sum_{k=0}^{N-1} f_d(q_k,q_{k+1})(\delta q_k,\delta q_{k+1}) = \sum_{k=0}^{N-1} \hat{f}_d(v_k,\tau_{k+1}) \left( (\delta v_k,\delta \tau_{k+1}) \oplus [(q_k,q_{k+1}) , \tilde{\xi}_k] \right),$$
  where $\tilde{\xi}_k = \frac{1}{2}(\xi_k + \xi_{k+1})$.
	
  Therefore, since
  $$d\hat{\frakS}_d(v_\cdot,\tau_\cdot)(\delta v_\cdot,\delta \tau_\cdot) + \sum_{k=0}^{N-1} \hat{f}_d(v_k,\tau_{k+1}) \left( (\delta v_k,\delta \tau_{k+1}) \oplus [(q_k,q_{k+1}) , \tilde{\xi}_k] \right) = 0,$$
  we have
  $$d\frakS_d(q_\cdot)(\delta q_\cdot) + \sum_{k=0}^{N-1} f_d(q_k,q_{k+1})(\delta q_k,\delta q_{k+1}) = 0.$$

  \underline{$1 \Lra 3$}. Given an infinitesimal variation
  $(\delta v_\cdot,\delta \tau_\cdot, \xi.)$ as in the statement, we can define
  $\delta q_k := h^{q_k}(\delta \tau_k) + (\xi_k)_Q(q_k)$,
  $1 \le k \le N$, together with the end point $\delta q_0 := 0$, so
  that
  $(\delta v_k,\delta \tau_{k+1}) = T_{(q_k,q_{k+1})} \Upsilon (\delta
  q_k,\delta q_{k+1})$. Note that the variations $\delta q_\cdot$ have
  vanishing end points. Indeed, $\delta q_0 = 0$ by construction and
  $$\delta q_N = h^{q_N} (\delta \tau_N) + (\xi_N)_Q(q_N) = h^{q_N} (0) + 0_Q(q_N) = 0.$$
	
  A computation analogous to that of $3 \Lra 1$ yields
  \begin{equation*}
    \begin{split}
      d\hat{\frakS}_d(\delta v_\cdot,\delta \tau_\cdot) &+ \sum_{k=0}^{N-1} \hat{f}_d(v_k,\tau_{k+1}) \left( (\delta v_k,\delta \tau_{k+1}) \oplus [(q_k,q_{k+1}) , \tilde{\xi}_k] \right) \\
      &= d\frakS_d(q_\cdot)(\delta q_\cdot) + \sum_{k=0}^{N-1}
      f_d(q_k,q_{k+1})(\delta q_k,\delta q_{k+1}) = 0.
    \end{split}
  \end{equation*}
	
  \underline{$3 \Longleftrightarrow 4$}. For infinitesimal variations
  $(\delta v_\cdot,\delta \tau_\cdot,\xi.)$, writing
  $\delta v_k = T_{(q_k,w_k)} \rho (\delta q_k , \delta w_k)$ and
  following the proof of Theorem 5.11 of \cite{F-T-Z}, we have
  \begin{equation}\label{dShat}
    \begin{split}
      d\hat{\frakS}_d(v_\cdot,\tau_\cdot)&(\delta v_\cdot,\delta \tau_\cdot) = \sum_{k=1}^{N-1} \left( D_1 \check{L}_d(q_k,w_k,\tau_{k+1}) \circ h^{q_k} + D_3 \check{L}_d(q_{k-1},w_{k-1},\tau_k) \right. \\
      & + D_2 \check{L}(q_k,w_k,\tau_{k+1}) D_1\calA_d(q_k,q_{k+1}) \circ h^{q_k} \\
      & + \left. D_2 \check{L}(q_{k-1},w_{k-1},\tau_k) D_2\calA_d(q_{k-1},q_k) \circ h^{q_k} \right) (\delta \tau_k) \\
      & + \sum_{k=0}^{N-1} \left( D_1\check{L}_d(q_k,w_k,\tau_{k+1}) \left( (\xi_k)_Q(q_k) \right) \right. \\
      & + \left. D_2\check{L}_d(q_k,w_k,\tau_{k+1}) T_{(q_k,q_{k+1})}
        \calA_d \left( (\xi_k)_Q(q_k) , (\xi_{k+1})_Q(q_{k+1}) \right)
      \right).
    \end{split}
  \end{equation}
	
  The same argument may be applied to the force $\check{f}_d$, yielding:
  \begin{equation*}
    \begin{split}
      \sum_{k=0}^{N-1} \check{f}_d(q_k,w_k,\tau_{k+1})&(\delta q_k,\delta w_k,\delta \tau_{k+1}) = \\
      &\sum_{k=1}^{N-1} \left( \check{f}_d^1(q_k,w_k,\tau_{k+1}) \circ h^{q_k} + \check{f}_d^3(q_{k-1},w_{k-1},\tau_k) \right. \\
      &+ \check{f}_d^2(q_k,w_k,\tau_{k+1}) D_1\calA_d(q_k,q_{k+1}) \circ h^{q_k} \\
      &+ \left. \check{f}_d^2(q_{k-1},w_{k-1},\tau_k) D_2\calA_d(q_{k-1},q_k) \circ h^{q_k} \right) (\delta \tau_k) \\
      &+ \sum_{k=0}^{N-1} \left( \check{f}_d^1(q_k,w_k,\tau_{k+1}) \left( (\xi_k)_Q(q_k) \right) \right. \\
      &+ \left. \check{f}_d^2(q_k,w_k,\tau_{k+1}) T_{(q_k,q_{k+1})}
        \calA_d \left( (\xi_k)_Q(q_k) , (\xi_{k+1})_Q(q_{k+1}) \right)
      \right).
    \end{split}
  \end{equation*}
	
  Putting both things together, since the variations $\delta \tau_\cdot$
  are independent of those generated by $\xi.$, we have that point $3$
  in the statement is equivalent to the vanishing of the summation
  involving $\delta \tau_\cdot$ and the one involving $\xi.$
  independently. The first of these conditions, since the
  $\delta \tau_\cdot$ are free, is equivalent to $\phi_{f_d}$ being
  identically zero.

  For the second summation, we can decompose the tangent map of
  $\calA_d$ and rearrange indexes to obtain
  \begin{equation*}
    \begin{split}
      \sum_{k=0}^{N-1} &\left( \check{f}_d^1(q_k,w_k,\tau_{k+1}) \left( (\xi_k)_Q(q_k) \right) \right. \\
      &+ \left. \check{f}_d^2(q_k,w_k,\tau_{k+1}) T_{(q_k,q_{k+1})} \calA_d \left( (\xi_k)_Q(q_k) , (\xi_{k+1})_Q(q_{k+1}) \right) \right) \\
      &= \sum_{k=0}^{N-1} \left( \check{f}_d^1(q_k,w_k,\tau_{k+1}) \left( (\xi_k)_Q(q_k) \right) \right. \\
      &+ \check{f}_d^2(q_k,w_k,\tau_{k+1}) D_1 \calA_d(q_k,q_{k+1}) \left( (\xi_k)_Q(q_k) \right) \\
      &+ \left. \check{f}_d^2(q_k,w_k,\tau_{k+1}) D_2 \calA_d (q_k,q_{k+1}) \left( (\xi_{k+1})_Q(q_{k+1}) \right) \right) \\
      &= \sum_{k=1}^{N-1} \left( \check{f}_d^1(q_k,w_k,\tau_{k+1}) \left( (\xi_k)_Q(q_k) \right) \right. \\
      &+ \check{f}_d^2(q_k,w_k,\tau_{k+1}) D_1 \calA_d \left( (\xi_k)_Q(q_k) \right) \\
      &+ \left. \check{f}_d^2(q_{k-1},w_{k-1},\tau_k) D_2 \calA_d
        (q_{k-1},q_k) \left( (\xi_k)_Q(q_k) \right) \right),
    \end{split}
  \end{equation*}
  where the first and last terms disappear as a consecuence of
  $\xi_0 = \xi_N = 0$. Finally, adding this to equation \eqref{dShat},
  we see that the vanishing of the summation involving $\xi.$ is
  equivalent to $\psi_{f_d}$ being zero for all $\xi_k$.
\end{proof}

\begin{remark}
  If $f_d = 0$, the statement of Theorem \ref{teor-red-forzada}
  reduces to the unconstrained version of Theorem 5.11 of
  \cite{F-T-Z}.
\end{remark}

\begin{example}\label{example-reduction}
  We return to the FDMS presented in Example
  \ref{example-discretization}, where $Q = \rr^2$. In what follows, we
  will consider $Q = \rr^2 \setminus \{ 0 \}$ so that $S^1$ acts on it
  by rotations without fixed points. Using polar coordinates, this
  configuration space is equivalent to $\rr^+ \times S^1$, where
  $\rr^+ := (0,+\infty)$.

  In order to study this system, we study its covering
  system\footnote{Given a forced mechanical system $M := (Q,L,f)$ and
    a covering map $\phi: \overline{Q} \lra Q$, it is possible to lift
    the system to a forced mechanical system
    $\overline{M} := (\overline{Q}, (T\phi)^*L, (T\phi)^*f)$. There is
    a correspondence between the trajectories of $M$ and those of
    $\overline{M}$ obtained using path lifting from $Q$ to
    $\overline{Q}$ and $\phi$ to go in the opposite direction.} with
  $Q = \rr^+ \times \rr$,
  \[
    \begin{split}
      L_d&(r_0,\eta_0,r_1,\eta_1) = \\
      &\frac{h}{2} \left(\frac{r_1-r_0}{h}\right)^2 + \frac{h}{2}
      \left(\frac{r_1+r_0}{2}\right)^2
      \left(\frac{\eta_1-\eta_0}{h}\right)^2 -
      \left(\frac{r_1+r_0}{2}\right)^2 \left(
        \left(\frac{r_1+r_0}{2}\right)^2-1 \right)^2,
    \end{split}
  \]
  and force $f_d = f_d^- + f_d^+$ given by
  \[
    \begin{split}
      f_d^-&(r_0,\eta_0,r_1,\eta_1) \\
      & = \underset{=: c_{r_0}}{\underbrace{\frac{k}{2} \left(
            -\frac{2}{h^2}(r_1-r_0)+\frac{r_1+r_0}{2}
            \left(\frac{\eta_1-\eta_0}{h}\right)^2 \right)}} dr_0
      \underset{=: c_{\eta_0}}{\underbrace{-
          \frac{k}{h^2}\left(\frac{r_1+r_0}{2}\right)^2
          (\eta_1-\eta_0)}} \ d\eta_0,
    \end{split}
  \]
  \[
    \begin{split}
      f_d^+&(r_0,\eta_0,r_1,\eta_1) \\
      & = \underset{=: c_{r_1}}{\underbrace{-\frac{k}{2} \left(
            \frac{2}{h^2}(r_1-r_0)+ \frac{r_1+r_0}{2}
            \left(\frac{\eta_1-\eta_0}{h}\right)^2 \right)}} dr_1
      \underset{=: c_{\eta_1}}{\underbrace{ - \frac{k}{h^2}
          \left(\frac{r_1+r_0}{2}\right)^2 (\eta_1-\eta_0)}} \
      d\eta_1.
    \end{split}
  \]
  That is,
  $f_d (r_0,\eta_0,r_1,\eta_1) =c_{r_0} dr_0+ c_{\eta_0} d\eta_0 +
  c_{r_1} dr_1 + c_{\eta_1} d\eta_1$.

  Whereas the original system had $S^1$ as a symmetry group, the
  covering system has $\rr$ as a symmetry group, acting by
  $l_g^Q(r,\eta) := (r,\eta + g)$. It is indeed a symmetry of the
  system, since $L_d$ is clearly invariant and we can check that $f_d$
  is $G$-equivariant: given
  $$((q_0,q_1),(\delta q_0,\delta q_1)) = (((r_0,\eta_0),(r_1,\eta_1)) , ((\delta r_0,\delta \eta_0),(\delta r_1,\delta \eta_1))) \in T(Q \times Q),$$
  we have
  \[
    \begin{split}
      l_g^{T^*(Q \times Q)}(f_d(q_0,q_1)) & (l_g^{Q \times Q}(q_0,q_1),(\delta q_0,\delta q_1)) \\
      &= f_d(q_0,q_1) \left( Tl_{-g}^{Q \times Q} \left( l_g^{Q \times Q}(q_0,q_1) , (\delta q_0,\delta q_1) \right) \right) \\
      &= f_d(q_0,q_1)((q_0,q_1) , (\delta q_0,\delta q_1)).
    \end{split}
  \]

  On the other hand,
  \[ f_d(l_g^{Q \times Q}(q_0,q_1)) (l_g^{Q \times Q}(q_0,q_1) ,
    (\delta q_0,\delta q_1)) = f_d(q_0,q_1)(\delta q_0,\delta q_1), \]
  where we have used that
  $\eta_1 + g - (\eta_0 + g) = \eta_1 - \eta_0$.
	
  $\ast$ The reduced spaces.
	
  Considering the discrete connection $\calA_d : Q \times Q \lra G$
  given by \\ $\calA_d((r_0,\eta_0),(r_1,\eta_1)) := \eta_1 - \eta_0$,
  we have
  \[ Q/G \simeq \rr^+, \quad \tilde{G} = (Q \times G)/G \simeq \rr^+
    \times \rr \] with $\pi : Q \lra Q/G$ and
  $\rho : Q \times G \lra \tilde{G}$ given by
  \[ \pi(r,\eta) = r \text{ and } \rho((r,\eta),g) = (r,g), \]
  \[ \tilde{G} \times Q/G \simeq \rr^+ \times \rr \times \rr^+ \] with
  $\Upsilon : Q \times Q \lra \tilde{G} \times Q/G$ given by
  \[ \Upsilon((r_0,\eta_0),(r_1,\eta_1)) = (r_0,\eta_1 - \eta_0,r_1). \]
	
  Since $\frakg := \Lie(G) \simeq \rr$, considering the principal
  connection $\frakA : TQ \lra \frakg$ given by
  $\frakA(\delta r,\delta \eta) := \delta \eta$, we have
  $$\tilde{\frakA}((\delta r_0,\delta \eta_0) , (\delta r_1,\delta \eta_1)) = \frac{1}{2}(\delta \eta_0 + \delta \eta_1)$$
  and, since the adjoint action is trivial,
  $$\tilde{\frakg} = (Q \times Q \times \frakg)/G \simeq \tilde{G} \times Q/G \times \frakg \simeq \rr^+ \times \rr \times \rr^+ \times \rr,$$
  with the identification
  $$[(r_0,\eta_0) , (r_1,\eta_1) , \xi] \mapsto (r_0 , \eta_1 - \eta_0 , r_1 , \xi).$$

  Therefore,
  $T^*(\tilde{G} \times Q/G) \oplus \hat{\frakg}^* \simeq T^*(\rr^+
  \times \rr \times \rr^+) \oplus \hat{\rr}^*$, where
  \[
    \hat{\rr} \simeq \rr^+ \times \rr \times \rr^+ \times \rr.
  \]
	
  $\ast$ $\hat{L}_d$ and $\hat{f}_d$.
	
  The reduced Lagrangian is
  \[
    \begin{split}
      \hat{L}_d&(v_0,\tau_1) = \hat{L}_d((r_0,g_0),r_1) \\
      &= \frac{h}{2} \left(\frac{r_1-r_0}{h}\right)^2 + \frac{h}{2}
      \left(\frac{r_1+r_0}{2}\right)^2 \left(\frac{g_0}{h}\right)^2 -
      \left(\frac{r_1+r_0}{2}\right)^2 \left(
        \left(\frac{r_1+r_0}{2}\right)^2-1 \right)^2
    \end{split}
  \]
  and, using Proposition \ref{f_d-hat},
  \[
    \begin{split}
      \hat{f}_d((r_0,g_0),r_1) &= \frac{k}{h^2} \left( (r_0 - r_1) + \frac{r_0 + r_1}{4} g_0^2 \right) \ dr_0 \\
      & \ \ \ + \frac{k}{h^2} \left( (r_0 - r_1) - \frac{r_0 + r_1}{4}
        g_0^2 \right) \ dr_1 - \frac{k}{h^2} \frac{(r_0 + r_1)^2}{2}
      g_0 \ \xi^*,
    \end{split}
  \]
  where $\xi^*$ is the basis of $\rr^*$.
\end{example}

\begin{remark}
  This system is the midpoint discretization of the Rayleigh system
  from Example 2 in \cite{dL-L-LG-22} and represents the movement of a
  unit mass particle in the plane with a radial potential (see Example
  2.3.2 in \cite{M-West}).
\end{remark}


\subsection{Reconstruction}

It is well known that once the dynamics of the reduced system has been
resolved, one is usually interested in reconstructing the original
dynamics.

Observe that the discrete curves
$(v_\cdot,\tau_\cdot) = ((v_0,\tau_1),\ldots,(v_{N-1},\tau_N))$ on
$\tilde{G} \times Q/G$ considered in Theorem \ref{teor-red-forzada}
satisfy $\tau_k = \pr^{Q/G}(v_k)$ for all $1 \le k \le N-1$.

\begin{theorem}\label{th:reconstruction}
  Let $\pi : Q \lra Q/G$ be a principal $G$-bundle and $\calA_d$ and
  $\frakA$ be an affine discrete connection and a principal connection
  on it, respectively. Let $(Q,L_d,f_d)$ be a FDMS with symmetry group
  $G$.

  Let $(v_\cdot,\tau_\cdot) = ((v_0,\tau_1),\ldots,(v_{N-1},\tau_N))$ be a
  discrete curve on $\tilde{G} \times (Q/G)$ such that
  $\tau_k = \pr^{Q/G}(v_k)$ for all $1 \le k \le N-1$.
  \begin{enumerate}
  \item \label{it:reconstruction-curve_lifting} Given $q_0 \in Q$ such
    that $\pi(q_0) = \pr^{Q/G}(v_0)$, there exists a unique discrete
    curve $q_\cdot = (q_0,\ldots,q_N)$ such that
    $\Upsilon(q_k,q_{k+1}) = (v_k,\tau_{k+1})$ for all
    $0 \le k \le N-1$.

  \item \label{it:reconstruction-trajectory_lifting} If $(v_\cdot,\tau_\cdot)$
    satisfies the conditions of item~\ref{it:teor-red-forzada-rem_q}
    of Theorem \ref{teor-red-forzada}, then the curve $q_\cdot$ of
    item~\ref{it:reconstruction-curve_lifting} is a trajectory of
    $(Q,L_d,f_d)$.
  \end{enumerate}
\end{theorem}

\begin{proof}
  \leavevmode \ref{it:reconstruction-curve_lifting}. For
  $(v_0,\tau_1) \in \tilde{G} \times Q/G$, using that
  $\Upsilon : Q \times Q \lra \tilde{G} \times Q/G$ is surjective, we
  can first take any $(\tilde{q}_0,\tilde{q}_1) \in Q \times Q$ such
  that $\Upsilon(\tilde{q}_0,\tilde{q}_1) = (v_0,\tau_1)$. Since
  $\pi : Q \lra Q/G$ is a principal $G$-bundle and
  $\pi(\tilde{q}_0) = \pr^{Q/G}(v_0) = \pi(q_0)$, there exists
  $g \in G$ such that $l_g^Q(\tilde{q}_0) = q_0$. Since $\Upsilon$ is
  also a principal $G$-bundle (see Remark \ref{Upsilon-bundle}), we
  can use $g$ to produce
  $(q_0,q_1) := l_g^{Q \times Q}(\tilde{q}_0,\tilde{q}_1)$. Notice
  that $(q_0,q_1)$ is the only lift of $(v_0,\tau_1)$ whose first
  component is $q_0$. Now, $q_1$ satisfies
  $\pi(\tilde{q}_1) = \pr^{Q/G}(v_1) = \pi(q_1)$, and so iterating
  this procedure defines a unique curve $q_\cdot$ on $Q$.

  \ref{it:reconstruction-trajectory_lifting}. It is enough to notice
  that the discrete curve $q_\cdot$ constructed applying the process
  described previously is a lift of $(v_\cdot,\tau_\cdot)$. Therefore, both
  curves are related as in the statement of Theorem
  \ref{teor-red-forzada}.
\end{proof}


\section{Applications}

In this section we first study the evolution of the discrete momentum
maps under conditions that are insufficient to apply the Discrete
Noether Theorem, proving that they evolve linearly along the
trajectories of the system.

The last part of this section is dedicated to the possibility of
having conserved structures on the reduced system. We show that if
there is a Poisson structure on $Q \times Q$ preserved by the flow of
a FDMS and if a symmetry group $G$ of the system acts by Poisson maps,
then there exists an induced Poisson structure on
$\tilde{G} \times Q/G$ preserved by the flow of the reduced system.


\subsection{Evolution of the momentum maps}

In the context of forced discrete mechanical systems, it is known that
in the presence of symmetries the Discrete Forced Noether Theorem
holds (see Theorem 3.2.1 of \cite{M-West} and Theorem 2 of
\cite{dL-L-LG-22}).

Following the exposition of \cite{dL-L-LG-22}, given a FDMS and a Lie
group $G$ acting on $Q$, we have the {\it forced discrete momentum
  maps} $J_{d,f}^\pm : Q \times Q \lra \frakg^*$ defined by
\begin{eqnarray*}
  \left\langle J_{d,f}^+ (q_0,q_1) , \xi \right\rangle & := & \left\langle D_2 L_d(q_0,q_1) + f_d^+(q_0,q_1) , \xi_Q(q_1) \right\rangle \\
  \left\langle J_{d,f}^- (q_0,q_1) , \xi \right\rangle & := & \left\langle -D_1 L_d(q_0,q_1) - f_d^-(q_0,q_1) , \xi_Q(q_0) \right\rangle.
\end{eqnarray*}

If, for some $\xi \in \frakg$
$\langle J_{d,f}^+ , \xi \rangle = \langle J_{d,f}^- , \xi \rangle$,
we can define a unique momentum map
$J_{d,f,\xi} : Q \times Q \lra \rr$ as
$$J_{d,f,\xi} (q_0,q_1) := \left\langle J_{d,f}^+(q_0,q_1) , \xi \right\rangle = \left\langle J_{d,f}^-(q_0,q_1) , \xi \right\rangle.$$

Using this notation, there exists the following version of Noether's
Theorem (see Theorem 3 of \cite{dL-L-LG-22}).

\begin{theorem}
  Let $(Q,L_d,f_d)$ be a FDMS. If $G$ is a Lie group acting on $Q$ in
  such a way that $L_d$ is $G$-invariant, then for each
  $\xi \in \frakg$,
  \begin{enumerate}
  \item $f_d(\xi_{Q \times Q}) = 0$ if and only if $J_{d,f,\xi}$ is
    well defined and is a constant of motion.
  \item When $f_d(\xi_{Q \times Q}) = 0$, $f_d$ is $\xi$-invariant
    (i.e., the Lie derivative $\calL_{\xi_{Q \times Q}} f_d = 0$) if
    and only if $i_{\xi_{Q \times Q}} d f_d = 0$.
  \end{enumerate}
\end{theorem}

Since it is not a hypothesis that $G$ be a symmetry group of the
system (in the sense of Definition \ref{sym-group-def}), it is natural
to ask ourselves if there is some kind of relation between those
conditions and the forces being $G$-equivariant.

\begin{lemma}
  Let $G$ be a Lie group acting on a smooth manifold $M$. If
  $\alpha : M \lra T^*M$ is a $G$-equivariant $1$-form, then $\alpha$
  is $\xi$-invariant for all $\xi \in \frakg$, i.e.,
  $\calL_{\xi_M} \alpha = 0$.
\end{lemma}

\begin{proof}
  Notice that being $G$-equivariant means that for all $g \in G$,
  $$l_g^{T^*M} \circ \alpha \circ l_{g^{-1}}^M = \alpha.$$
	
  For any $\xi \in \frakg$, the flow of the vector field $\xi_M$ is
  given by $F_t^{\xi_M}(x) = l_{\exp(t\xi)}^M (x)$ (recalling
  \eqref{inf-generator}).
	
  Then, given $v_x \in T_xM$,
  \begin{equation*}
    \begin{split}
      (F_t^{\xi_M})^*( \alpha (F_t^{\xi_M} (x)) )(v_x) &= \alpha(l_{\exp(t\xi)}^M (x)) \left( T_x l_{\exp(t\xi)}^M (v_x) \right) \\
      &= \left( l_{\exp(t\xi)^{-1}}^{T^*M} \alpha \left(
          l_{\exp(t\xi)}^M (x) \right) \right) (v_x) = \alpha(x)(v_x),
    \end{split}
  \end{equation*}
  where in the last equality we have used the $G$-equivariance of
  $\alpha$. Therefore,
  $$(F_t^{\xi_M})^*( \alpha (F_t^{\xi_M} (x)) ) - \alpha(x) = 0$$
  and
  $$\calL_{\xi_M} \alpha = \lim_{t \to 0} \frac{(F_t^{\xi_M})^*( \alpha (F_t^{\xi_M} (x)) ) - \alpha(x)}{t} = \lim_{t \to 0} \frac{0}{t} = 0.$$
\end{proof}

Applying the last lemma to the force $f_d$, we have the following result.

\begin{proposition}
  Let $(Q,L_d,f_d)$ be a FDMS with symmetry group $G$. Then, $f_d$ is
  $\xi$-invariant for all $\xi \in \frakg$.
\end{proposition}

When there is no unique momentum map, we might wonder how the two maps
we do have evolve. To state the result we will use the following
expected definition: the two maps
$J_{d,f,\xi}^\pm : Q \times Q \lra \rr$ are defined by
\begin{eqnarray*}
  J_{d,f,\xi}^+ (q_0,q_1) & := & \left\langle J_{d,f}^+ (q_0,q_1) , \xi \right\rangle \\
  J_{d,f,\xi}^- (q_0,q_1) & := & \left\langle J_{d,f}^- (q_0,q_1) , \xi \right\rangle.
\end{eqnarray*}

\begin{proposition}\label{Jd-evol}
  Let $(Q,L_d,f_d)$ be a FDMS with symmetry group $G$ and $Q$
  connected. If $df_d$ is horizontal, then the quantities
  $J_{d,f,\xi}^\pm$ evolve according to
\begin{eqnarray*}
  J_{d,f,\xi}^+(q_1,q_2) & = & J_{d,f,\xi}^+(q_0,q_1) + \mu \\
  J_{d,f,\xi}^-(q_1,q_2) & = & J_{d,f,\xi}^-(q_0,q_1) + \mu,
\end{eqnarray*}
for some constant $\mu \in \rr$, where $(q_1,q_2)$ is constructed from
$(q_0,q_1)$ by the flow of the system.
\end{proposition}

\begin{proof}
  By the previous proposition, given $\xi \in \frakg$,
  $$0 = \calL_{\xi_{Q \times Q}} f_d = i_{\xi_{Q \times Q}} df_d + d i_{\xi_{Q \times Q}} f_d = d \left( f_d(\xi_{Q \times Q}) \right).$$
	
  Then, as $Q \times Q$ is connected, $f_d(\xi_{Q \times Q}) = -\mu$ for
  some constant $\mu \in \rr$, so that
  \begin{equation}\label{Jd-Jd}
    -\mu = \langle dL_d + f_d , \xi_{Q \times Q} \rangle = J_{d,f,\xi}^+ - J_{d,f,\xi}^-,
  \end{equation}
  where we are using that $dL_d(\xi_{Q \times Q}) = 0$ because $L_d$
  is $G$-invariant.
	
  On the other hand, evaluating the variational principle on a
  trajectory $q_\cdot$ and infinitesimal variations
  $\delta q_k := \xi_Q(q_k)$, $0 \le k \le N$, we have
  $$0 = \sum_{k=0}^{N-1} (dL_d + f_d)(q_k,q_{k+1})(\xi_{Q \times Q}(q_k,q_{k+1})) = \left\langle J_{d,f}^+(q_{N-1},q_N) - J_{d,f}^-(q_0,q_1) , \xi \right\rangle.$$
	
  Taking $N=2$ and switching the notation,
  $$0 = J_{d,f,\xi}^+(q_1,q_2) - J_{d,f,\xi}^-(q_0,q_1).$$
	
  Combining this equation with (\ref{Jd-Jd}),
  $$J_{d,f,\xi}^+(q_1,q_2) = J_{d,f,\xi}^-(q_0,q_1) \stackrel{(\ref{Jd-Jd})}{=} J_{d,f,\xi}^+(q_0,q_1) + \mu.$$
	
  For the other momentum map,
  $$J_{d,f,\xi}^-(q_1,q_2) \stackrel{(\ref{Jd-Jd})}{=} J_{d,f,\xi}^+(q_1,q_2) + \mu = J_{d,f,\xi}^-(q_0,q_1) + \mu.$$
\end{proof}

\begin{example}
  We analyze the discrete version of an example proposed in \cite{M15}
  (Section 2.2).

  Let us consider a disk of mass $m$ and radius $r$ placed on a
  horizontal surface, whose movement is affected by a certain friction
  force. The configuration of this system is $Q = S^1$, but just as we
  did in Example \ref{example-reduction} we work with its covering
  system with configuration space $Q = \rr$,
  \[ L_d(\varphi_0,\varphi_1) := \frac{mr^2}{4}
    \frac{(\varphi_1-\varphi_0)^2}{h} \] and force
  $f_d = -\dfrac{\eta mgr}{2h} (d\varphi_0+d\varphi_1)$, where $\eta$
  is a constant friction coefficient.

  Notice that the trajectories of the system are given by
  \[ \varphi_{k+1} = 2\varphi_k - \varphi_{k-1} - \frac{4\eta
      g}{r}. \]

  Whereas the group $S^1$ is a symmetry of the original system,
  $G := \rr$ is a symmetry of the covering system with
  $l_g^Q(\varphi) := \varphi + g$. For $\xi \in \frakg \simeq \rr$,
  the momentum maps are given by
  \[ \begin{split}
      J_{d,f,\xi}^+(\varphi_0,\varphi_1) 
      & = \left( \frac{mr^2}{2h}(\varphi_1-\varphi_0) - \frac{mr\eta g}{2h}  \right) \xi \\
      J_{d,f,\xi}^-(\varphi_0,\varphi_1) &= \left(
        \frac{mr^2}{2h}(\varphi_1-\varphi_0) + \frac{mr\eta g}{2h}
      \right) \xi.
    \end{split} \]

  Now, considering the first step of a trajectory of the FDMS, we
  have
  $$(\varphi_0,\varphi_1,\varphi_2)=\left(\varphi_0,\varphi_1,2\varphi_1
    - \varphi_0 - \dfrac{4\eta g}{r}\right).$$

  Then,
  \[ \begin{split}
      J_{d,f,\xi}^+(\varphi_0,\varphi_1) & = \left( \frac{mr^2}{2h}(\varphi_1-\varphi_0) + \frac{mr\eta g}{2h}  \right) \xi \\
      J_{d,f,\xi}^+(\varphi_1,\varphi_2) & = \left( \frac{mr^2}{2h}(\varphi_2-\varphi_1) + \frac{mr\eta g}{2h}  \right) \xi \\
      & =J_{d,f,\xi}^+(\varphi_0,\varphi_1) - \frac{2mr\eta g}{h} \xi
    \end{split} \]
  and, similarly,
  \[
    J_{d,f,\xi}^-(\varphi_1,\varphi_2) =
    J_{d,f,\xi}^-(\varphi_0,\varphi_1) - \frac{2mr\eta g}{h} \xi.
  \]

  Therefore, although Noether's Theorem does not apply, we see that
  the momentum maps evolve in a controlled fashion, with the constant
  $\mu$ of Proposition \ref{Jd-evol} being
  $\mu = -\frac{2mr\eta g}{h} \xi$. Notice that the proposition can
  indeed be applied, since $df_d = 0$, implying that it is horizontal.
\end{example}


\subsection{Conservation of Poisson structures}

We follow the ideas found in Section 8 of \cite{F-T-Z-16}. It is a
well known and used fact that if $(Q,L_d)$ is a regular discrete
mechanical system (without forces), there is a symplectic structure
$\omega_{L_d}$ defined on (an open subset containing the diagonal of)
$Q \times Q$. Furthermore, the flow of the system preserves
$\omega_{L_d}$. In our setting of forced discrete mechanical systems,
the existence of such a symplectic structure is no longer guaranteed,
although, there are systems that posses a symplectic structure
preserved by the flow\footnote{This is the case, for example, if the
  force $f_d$ satisfies $df_d = \pr_2^* \beta - \pr_1^*\beta$ for some
  $2$-form $\beta$ on $Q$. This scenario arises naturally in the
  context of the discrete Routh reduction.}.

If we have a symmetric FDMS $(Q,L_d,f_d)$, it is a natural question
whether a preserved symplectic structure induces a $2$-form on
$\tilde{G} \times Q/G$ preserved by the flow of the reduced system. As
the space $\tilde{G} \times Q/G$ may not have even dimension, it may
be impossible for it to be a symplectic manifold to begin
with. Therefore, it seems convenient to focus on a more general
scenario and consider Poisson structures.

Recall that a smooth map
$F : (P_1,\{ \cdot , \cdot \}_1) \lra (P_2,\{ \cdot , \cdot \}_2)$
between Poisson manifolds is a {\it Poisson map} if for every pair of
smooth maps $g,h : P_2 \lra \rr$,
$$F^* \{ g,h \}_2 = \{ F^* g , F^*h \}_1.$$

\begin{proposition}
  Let $\pi : Q \lra Q/G$ be a principal $G$-bundle and $\calA_d$ and
  $\frakA$ be an affine discrete connection and a principal connection
  on it, respectively. Let $(Q,L_d,f_d)$ be a FDMS with symmetry group
  $G$ acting by Poisson maps. If $\{ \cdot , \cdot \}$ is a Poisson
  structure on $Q \times Q$ preserved by the flow $\bF$ of the system,
  then there exists a Poisson structure $\{ \cdot , \cdot \}_r$ on
  $\tilde{G} \times Q/G$ preserved by the flow $\hat{\bF}$ of the
  reduced system such that $\Upsilon$ (defined in \eqref{red-spaces})
  is a Poisson map.
\end{proposition}

\begin{proof}
  As noted in Remark \ref{Upsilon-bundle},
  $\Upsilon : Q \times Q \lra \tilde{G} \times Q/G$ is a principal
  $G$-bundle, so that the action of $G$ on $Q \times Q$ is a free and
  proper surjective submersion. As $G$ acts on $Q \times Q$ by Poisson
  maps, it follows from Theorem 10.5.1 in \cite{M-R} that there is a
  unique Poisson structure $\{ \cdot , \cdot \}_r$ on
  $\tilde{G} \times Q/G$ such that $\Upsilon$ becomes a Poisson
  map. We want to show that this structure is preserved by the flow
  $\hat{\bF}$ of the reduced system. Theorem \ref{teor-red-forzada}
  yields the following commutative diagram:
  $$\xymatrixcolsep{5pc}\xymatrixrowsep{4pc}\xymatrix{
    Q \times Q \ar[r]^-{\bF} \ar[d]_-\Upsilon & Q \times Q \ar[d]^-\Upsilon \\
    \tilde{G} \times Q/G \ar[r]_-{\hat{\bF}} & \tilde{G} \times Q/G \\
  }$$ Since $\Upsilon \circ \bF$ and $\Upsilon$ are Poisson maps,
  $\hat{\bF}$ is also a Poisson map, by Lemma 8.3 in \cite{F-T-Z-16}.
\end{proof}


\section*{Acknowledgments}
This document is the result of research partially supported by grants
from the Universidad Nacional de Cuyo [codes 06/C009-T1 and 06/C567],
Universidad Nacional de La Plata, [codes X758, X915] and CONICET.


\providecommand{\bysame}{\leavevmode\hbox to3em{\hrulefill}\thinspace}
\providecommand{\MR}{\relax\ifhmode\unskip\space\fi MR }
\providecommand{\MRhref}[2]{%
	\href{http://www.ams.org/mathscinet-getitem?mr=#1}{#2}
}
\providecommand{\href}[2]{#2}

\end{document}